\documentclass[12pt,reqno]{article}
\usepackage{amssymb, amsmath}
\usepackage{amssymb,amsthm}
\newtheorem{Theorem}{Theorem}
\newtheorem{Remark}{Remark}

\newtheorem{Definition}{Definition}
\numberwithin{Definition}{section}
\numberwithin{Theorem}{section}
\numberwithin{Lemma}{section}
\numberwithin{Proposition}{section}

\numberwithin{Corollary}{section}
\numberwithin{Example}{section}
\numberwithin{Remark}{section}
\begin{document}
\title{Some remarks on weak generalizations of minima and quasi efficiency}
\author{Triloki Nath\thanks{Department of Mathematics and Statistics, Dr. Harisingh Gour Vishwavidyalaya, Sagar, Madhya Pradesh-470003, INDIA, Email- tnverma07@gmail.com} }

\date{}	
\maketitle

\begin{abstract}
  In this note, we remark, with sufficient mathematical rigor, that many weak generalizations of the usual minimum available in the literature are not true generalizations. Motivated by the  Ekeland Variational Principle, we provide, first time, the criteria for weaker generalizations of the usual minimum. Further, we show that the quasi efficiency, recently used in Bhatia et al. (Optim. Lett. 7, 127-135 (2013)) and introduced in  Gupta et al. ( Bull. Aust. Math. Soc. 74, 207-218 (2006)) is not a true generalization of the usual efficiency. Since the former paper relies heavily on the results of later one, so we discuss the later paper. We show that the necessary optimality condition is a consequence of the local Lipschitzness and sufficiency result is trivial in the later paper. Consequently, the duality results of the same paper are also inconsistent.
  \end{abstract}
\section{Introduction}\label{sec:intro}
 The notions of  maxima and minima of a real-valued function are ubiquitous in analysis and optimization. So, naturally, researchers  are interested in their generalizations. What properties should have the weaker generalizations of the usual minimum? is a fundamental question in mathematics at large.\\

\hspace*{0.5cm} The aim of this note is two fold, first we answer the question of weaker generalizations of the usual minimum and then we discuss some recent generalizations of the usual minimum and the usual efficiency. Any true weaker generalization is one which reduces to its original form under the original circumstances. In other words, a true weaker form is one which is implied by the original notion, and if it differs at some point, then it would not imply any condition much stronger than the implications by the original notion. For example, suppose a standard notion implies local Lipschitz continuity then a weaker notion would not imply differentiability but may imply almost everywhere differentiability (thanks to Rademacher's theorem). Thus, in the generalization process, some features of original notions are getting involved. Consequently, sometimes, we compromise for weaker generalizations with some seemingly stronger conditions but equivalent to the implicit conditions implied by the original one. These points are also clarified in section  \ref{sec:wglm}. Some recent generalizations in analysis and optimization are unable to recover the original. In fact, some of these exact points have already been noticed in \cite{Borwein16, Zalinescu14}. To be more explict, we quote the words of Borwein \cite{Borwein16}:
\begin{center}
``\emph{A generalisation is an extension of a known result to cover more cases (a unification), and/or weaken hypotheses and/or strengthen conclusions. The role of examples and counter-examples is to validate and substantiate one or more of these three roles}."
\end{center}
\hspace*{0.5 cm} Recently, Bhatia et al. \cite{Bhatia13} derived sufficient optimality conditions for quasi efficient solutions employed on some  new generalized convexity assumptions. Since the results established in the paper \cite{Bhatia13} heavily rely on the results of  Gupta at al. \cite{Bhatia06} pioneering  a weaker generalized efficiency called quasi efficiency for  vector optimization problems. Further,  in order to understand whether the results of  \cite{Bhatia13} make sense/correct, it is necessary that the notions and results of base paper \cite{Bhatia06} must be meaningful and correct. Therefore we confine our discussions on the source paper \cite{Bhatia06} only. We show  that the notions of quasi efficiency and local quasi efficiency are not relevant and are merely the consequences of Lipschitz and local Lipschitz  continuity. Theorem 2  of  \cite{Bhatia06} has a gap, which has been corrected as Theorem \ref{thm2} in the present note. It seems that the lower semicontinuity assumption is missing  in Theorem 3 of \cite{Bhatia06}, and we show that the assumption makes it a trivial result. Consequently, the duality results are affected. \\

Rest of this note is organized as follows. Section \ref{prel} contains some basic notions from the nonsmooth analysis. In section \ref {sec:wglm}, we define the appropriate criteria for a weaker generalization of minimum with sufficient mathematical support, and we compare the approximate solutions of \cite{Lorid82} with our criteria. The quasi efficiency defined in \cite{Bhatia06} is shown to be a consequence of Lipschitz continuity and a result is corrected in section \ref{sec:mainresult}. Finally, some conclusions are drawn in the last section.
\section{Preliminaries}
\label{prel}
\hspace*{0.5 cm} Throughout this note, $X$ will denote a nonempty subset of $\mathbb R^n$ and  $B(x_\circ,\delta):=\{x \in \mathbb{R}^n |~ \|  x-x_\circ\|< \delta\}$ is open ball. For a better understanding, we recall the following definitions and results from \cite{Clarkebk1}.
\begin{Definition}\rm
A function $h: X \rightarrow \mathbb R\cup \{+\infty\}$ is said to be  \textit{lower semicontinuous} at $x_\circ \in X$, if for every $\epsilon > 0$ there exists $\delta(x_\circ,\epsilon)> 0$ such that
$$h(x)\geqslant h(x_\circ)- \epsilon~ ~ ~\forall ~x \in B(x_\circ,\delta)\cap X. $$
\end{Definition}
 \begin{Definition}\rm
A function $h: X \rightarrow \mathbb R$ is said to be \textit{ locally Lipschitz} at $x \in X$, if there exist $L (x) \geqslant 0$ and  $\delta > 0$ such that
$$|h(y) - h(z)| \leqslant L \|y-z\|~ ~ ~\forall ~y, z  \in B(x,\delta)\cap X. $$
The smallest such $L$ is called (best) local Lipschitz constant for $h$ at $x$. Obviously, every number greater than the local Lipschitz constant works in the above inequality.\\
  \hspace*{0.5 cm} The function $h$ is said to be Lipschitz over $X$ if there exists $L \geqslant 0$(independent of the point) such that the above inequality holds for all $y, z  \in  X.$
  \end{Definition}
\begin{Definition} \rm
  Let  $h: X \rightarrow \mathbb R$ be a locally Lipschitz function at $x \in X$. The \textit{Clarke generalized directional derivative }of $h$ at the point $x$ and in the direction $v\in \mathbb R^n$, denoted by $h^\circ(x; v)$, is given by
$$h^\circ(x; v)=\limsup_{{z\rightarrow x},~ {\lambda\downarrow 0}} \frac{h (z+\lambda v)- h(z)}{\lambda} $$

and the \textit{Clarke generalized subdifferential} of $h$ at $x$, denoted by $\partial h(x)$, is given by
$$\partial h(x)= \left\{d \in \mathbb R^n| h^\circ(x; v)\geqslant  d^tv~ \forall v \in \mathbb R^n\right\}.$$

\end{Definition}
 \begin{Definition}\rm
  The \textit{Clarke tangent cone}  of $X$ at $x$ is a closed convex cone given as
$$T_X(x)= \left\{v \in X|~ d^{\circ}_X (x;v)= 0 \right\},$$
where $d_X$ is distance function on $\mathbb R^n$ defined by $d_X (y)= \inf \left\{\left\|y-x\right\| |~ x\in X\right\}$, which is globally Lipschitz.\\
The \textit{Clarke normal cone } of $X$ at $x$ is (negative) polar cone of $T_X(x)$ given as
$$N_X(x)=(T_X(x))^{-}:= \left\{d \in \mathbb R^n|~ d^t v \leqslant 0 ~~\forall~ v\in T_X(x)\right\}.$$
  \end{Definition}
  \begin{Remark}\rm
  By definition of Clarke directional derivative, it follows that  $d^{\circ}_{X \cap B(x,\delta)} (x;v)=d^{\circ}_X (x;v)$ for all $x\in X$. Consequently, $N_{X \cap B(x,\delta)}(x)= N_{X}(x)$ holds for all $x\in X$.
  \end{Remark}

\section{Weak Generalizations of Minimum}
\label{sec:wglm}
Let us recall that a point $x_{\circ}\in \mathbb{R}^n$ is said to be a point of the minimum of $f:\mathbb{R}^n \rightarrow\mathbb{R}\cup \{+\infty\}$ if $f(x)\geqslant f(x_{\circ})$ for all $x \in \mathbb{R}^n$. Consequently, $f(x)\geqslant f(x_{\circ})-\epsilon$ for all $x \in \mathbb{R}^n$ and for all $\epsilon>0$. Hence, we first note that $f$ is always lower semicontinuous at the point of minima. Secondly, the existence of minimum ensures for $f$ to be bounded below,  with  $f(x_\circ)$ as the greatest lower bound. This simple observation might be the main motivation to Ekeland for the following  elegant result, see \cite{Ekeland}.
\begin{Theorem}\rm (\textbf{Ekeland Variational Principle or EVP})
\label{thm2}
Let  $f:\mathbb{R}^n \rightarrow\mathbb{R}\cup \{+\infty\}$ be a lower semicontinuous function which is bounded below and is not identically $+\infty$. Let $\epsilon > 0$ and $x_{\circ}\in \mathbb{R}^n$ be such that $f(x_{\circ})\leqslant \displaystyle\inf_{x \in \mathbb{R}^n}f(x)+\epsilon$. Then, for any $\lambda>0$ there exists $x_{\lambda} \in \mathbb{R}^n$ such that
\item ~~{\rm (i)}$~~ f(x_{\lambda}) \leqslant f(x_{\circ})$.
\item ~~{\rm (ii)}$~~  \|x_{\lambda}-x_{\circ}\| \leqslant \lambda.$
\item ~~{\rm (iii)}$~~  f(x_{\lambda}) \leqslant f(x) + \dfrac{\epsilon}{\lambda}\|x- x_{\lambda}\|~~~ \forall~ x \in \mathbb R^n$.
\end{Theorem}
Note:- If $f:= + \infty$ or $\epsilon =0$, then the statement is trivial.\\

Before discussing the generalizations motivated by EVP, we must focus on the two simple properties of usual minimum: (i) lower semicontinuity, and (ii) greatest lower bound. In what follows, the following may be the criteria to qualify for a weaker minimum.

\begin{Definition} \rm
\label{wglm}
(\textbf{Weaker Generalized Minimum} ) A  point $x_{\circ}\in \mathbb{R}^n$ is said to be a \textit{weaker generalized  minimum} or WGM of $f:\mathbb{R}^n \rightarrow\mathbb{R}\cup \{+\infty\}$, if the following hold\\
(i) If $x_\circ$ is usual minimum then it is WGM.\\
(ii) The WGM implies that $f$ has a finite lower bound.\\
(iii) The WGM may imply lower semicontinuity at $x_\circ$, but not continuity.\\
(iv) If $l < f(x_\circ)$, then $l$ cannot be an upper bound for $f$ in any deleted neighbourhood of $x_\circ$.
\end{Definition}
\begin{Remark}
\label{rmrk-wgmm}\rm
 The condition (iii) of WGM is compatible with the class of lower semicontinuous functions which is the usual assumption on a class of functions in minimization problems.
\end {Remark}

Consider the following scalar problem (MP):\\
$~~~~~~~~~~~~~~~~~~~~~~~~ \min f(x),~ ~~~{\rm subject~ to}~~~ x \in X$ \hfill(MP)\\
$~~~~~~$where $f:\mathbb{R}^n \rightarrow\mathbb{R}\cup \{+\infty\}$.\\


To the best of our knowledge, the following WGM, generally referred to as \textit{approximate minimum} in the literature, is the first notion due to Loridan\cite{Lorid82} called $\epsilon-$\textit{solution} (we call it $\epsilon-$\textit{minimum}) for $(MP)$.  This notion is of course  motivated by the assumption in EVP.
\begin{Definition}\rm  \cite{Lorid82}.
\label{approx-min}
Let $\epsilon \geqslant 0$ be given, then a point $x_{\circ}(\epsilon) \in X$ is called an $\epsilon-$\textit{minimum} of $(MP)$ if
\begin{eqnarray}	
	f(x_{\circ}) \leqslant f(x) + \epsilon~~~\forall~x \in X.
\end{eqnarray}
The $x_{\circ}(\epsilon)$ signifies that  $x_{\circ}$ depends on $\epsilon$.
\end{Definition}
\begin{Remark}
\label{rmrk-1}
 If $\epsilon =0$, then $\epsilon-$minimum becomes the usual notion of minimum.
\end {Remark}
\begin{Remark}\rm
\label{rmrk-2}
 If $f$ is lower semicontinuous then one can verify that $\epsilon-$minimum qualifies for  WGM given in Definition \ref{wglm}. On the other hand, if $f$ is not lower semicontinuous then $\epsilon-$minimum may violate condition (iv) of WGM, e.g. consider $f: \mathbb R\rightarrow  \mathbb R$ defined by  $f(0)= 1$ and $f(x)=0$ otherwise, is not lower semicontinuous at $x_{\circ}=0$, but it is a $2-$minimum. Clearly, for $l=\frac{1}{2}$, there is a deleted neighbourhood of $x_\circ$ where $l$ is an upper bound for $f$. Thus, the condition (iv) is violated in this case. Hence, condition (iv) is an essential component for weaker generalization, which can not be violated by the lower semicontinuous functions.
\end {Remark}
The condition (iii) of EVP motivated Loridan \cite{Lorid82} to introduce $\epsilon-$quasi minimum (originally, $\epsilon-$\emph{quasi solution}).

\begin{Definition}
\label{qmin} \rm \cite{Lorid82}.
Let $\epsilon \geqslant 0$ be given, then a point $x_{\circ}(\epsilon) \in X$ is said to be an $\epsilon-$\textit{quasi minimum} of $(MP)$ if
\begin{eqnarray}
\label{qmn}	
	f(x_{\circ}) \leqslant f(x) + \sqrt{\epsilon} \| x- x_{\circ} \|~~~\forall~x \in X.
\end{eqnarray}
\end{Definition}

Let us examine Definition \ref{qmin} for WGM.
\begin{Remark}\rm
\label{rmrk-3}
For a given $\epsilon >0$, the $\epsilon-$quasi minimum does not qualify for WGM. Indeed,  the condition of  $\epsilon-$quasi minimum implies lower semicontinuity. For, let $\epsilon_1>0$ be arbitrary then we can choose $\delta < \dfrac{\epsilon_1}{\sqrt{\epsilon}}$, so that (\ref{qmn}) implies
\begin{eqnarray}
\label{qmn-1}	
	f(x_{\circ} ) \leqslant f(x) + \epsilon_1 ~~\forall~x \in X \cap B(x_\circ, \delta)
\end{eqnarray}
Thus $f$ is lower semicontinuous at $x_\circ$.
On the other hand, consider  $f(x)= - \sqrt{x}$ for $x \geqslant 0$ and $f(x)= -x$ for $x < 0$, is continuous at $x_{\circ}=0$, but not an $\epsilon-$quasi minimum point for $f$, for any $\epsilon>0$.\\
Nevertheless, we can not say that condition (iii) of WGM is violated because $\epsilon-$quasi minimum  does not imply continuity, e.g. the function $f: \mathbb R \rightarrow \mathbb R$ defined by $f(x):=x$ for $x \leqslant 0$ and $f(x):=1$ otherwise, has 1-quasi minimum. Thus, it is easy to see that all the  conditions of  WGM except (ii)  are satisfied for $\epsilon-$quasi minimum.\\
\hspace*{0.5cm}For $\epsilon>0$, consider  $f: \mathbb R\rightarrow  \mathbb R$ defined by $f(x)= x$ for $x \geqslant 0$ and $f(x)= \sqrt{\epsilon} ~x$ for $x < 0$, then $x_{\circ}=0$, is an $\epsilon-$quasi minimum point for $f$. But, clearly, $f$ is not bounded below. Hence, condition (ii) of WGM is violated.\\
\hspace*{0.5cm}Thus, $\epsilon-$quasi minimum alone does not share the property of being  $f$ bounded below, which is the main characteristic of the usual minimum. Thus, from the above discussions, it is clear that $\epsilon-$quasi minimum alone is not a suitable generalization of the usual global minimum. Hence, $\epsilon-$quasi minimum does not qualify for  WGM.
\end {Remark}
\begin{Remark}\rm
\label{rmrk-6}
By definition, it follows that every point of the domain of a lower semicontinuous function is a local  $\epsilon-$ minimum. Hence, the local $\epsilon-$minimum can not be a characteristic of a minimization problem. Instead, the global $\epsilon-$minimum makes sense and is a true generalization of the usual minimum for a class of lower semicontinuous function.
\end {Remark}

It may be noted that the Remark 3.1 of \cite{Lorid82} is superfluous, which says that every $\epsilon-$quasi minimum is  a local $\epsilon-$minimum: not a surprising one, it is trivial. Indeed, if $f$ is lower semicontinuous then every point is a local $\epsilon-$minimum. So, the $\epsilon-$quasi minimum is  also a local $\epsilon-$minimum. One may insist that function need not be lower semicontinuous, it is funny because we already have pointed out in Remark \ref{rmrk-3} that the $\epsilon-$quasi minimum implies lower semicontinuity at that point. In minimization problems, the weakest,  we can do,  is lower semicontinuity and to ensure the infimum (not necessarily minimum) finite lower bound condition is must for $f$. Thus, using EVP the $\epsilon-$quasi minimum exists near global $\epsilon-$minimum. Therefore, considering only $\epsilon-$quasi minimum which need  not be global $\epsilon-$minimum is not significant in minimization problems. That might be the reason for Loridan \cite{Lorid82} to consider the \emph{regular approximate $\epsilon-$solution} (see Definition 3.3 of \cite{Lorid82}) which is $\epsilon-$minimum and $\epsilon-$quasi minimum both. We must note that the $\epsilon-$minimum implies finite lower bound of the function whereas the $\epsilon-$quasi minimum implies the lower semicontinuity at that point. It is easy to verify that the regular approximate $\epsilon-$solution is WGM.\\\\
We close this section with the final remark.
\begin{Remark}\rm
\label{rmrk-4}
If the function $f$ is lower semicontinuous and bounded below then $\epsilon-$quasi minimum exists by EVP. We must note, that  the conclusion of EVP is not the condition (iii) only, the condition (i)  and (ii) can not be ignored in any generalization. Most importantly condition (ii), which says that $x_\lambda$ must be at most $\lambda$ distant  from the $\epsilon-$ minimum, is missing in $\epsilon-$quasi minimum.
\end {Remark}


  \section{Quasi efficiency and Main Results}\label{sec:mainresult}
In \cite{Bhatia06}, Authors consider the following vector optimization problem $(VP)$:
$$
\left.%
\begin{array}{lcl}
\min~f(x)=(f_1(x),..,f_p(x))\\
{\rm subject~ to }~g_j(x) \leqslant 0, j = 1,...m,
\end{array}%
\right\}~~~~~~~~~~~~~~~~~~~~~~~~~~~~~~~~~~~~~~~~~~~~~(VP)
$$
where $ f_i, g_j :X\rightarrow \mathbb{R},~i=1,...p, j=1,..,m,$  are functions on a nonempty subset $X$ of  $\mathbb R^n$.  Let us denote the feasible set of $(VP)$ by $\Omega:=\left\{x \in X | g_j(x) \leqslant 0, j=1,..m\right\}$.\\
\hspace*{0.5 cm} Recall the following from \cite{Bhatia06}.
\begin{Definition}\rm
 A feasible point $x_{\circ} \in \Omega$  is said to be a \textit{local efficient solution}  of $(VP)$ if there exists $\delta >0$ such that for any $x \in \Omega \cap B(x_{\circ},\delta)$ the following can not hold
 \begin{eqnarray*}
\label{les}	
	f_i(x) &\leqslant& f_i(x_{\circ})   ~~~\forall~  i=1,2,...,p,\\
	f_r(x) & < & f_r(x_{\circ})   ~~~{\rm for ~some ~} r.
\end{eqnarray*}
\end{Definition}

The following weak generalized efficiency is probably motivated by vector notion of $\epsilon-$ quasi minimum \cite{Lorid82}.
\begin{Definition} \rm \cite{Bhatia06}.
\label{quasi-1}
 A feasible point $x_{\circ} \in \Omega$  is said to be a \textit{local quasi efficient solution}  (or  \textit{quasi efficient solution} ) of $(VP)$ if there exist $\alpha \in int(R^p_+)$ and $\delta >0$ such that for any $x \in \Omega \cap B(x_{\circ},\delta)$ (or $x \in \Omega$) the following can not hold
 \begin{eqnarray*}
\label{les}	
	f_i(x) &\leqslant& f_i(x_{\circ}) -\alpha_i \|x-x_{\circ}\|  ~~~\forall~  i=1,2,...,p,\\
	f_r(x) & < & f_r(x_{\circ})-\alpha_r\|x-x_{\circ}\|   ~~~{\rm for ~some ~} r.
\end{eqnarray*}
\end{Definition}
\begin{Remark}{\rm
\label{rmrk1}\textbf{(Local) Lipschitz continuity $\Longrightarrow$ (local) quasi efficiency:}\\
If $f_i$ are locally Lipschitz at $x_{\circ} \in \Omega$, then it is a local  quasi efficient solution of $(VP).$ For, let $L_i \geqslant 0$ are local Lipschitz constant of $f_i$, then there exist  $\delta_i > 0$ such that
$$|f_i(x) - f_i(x_{\circ})| \leqslant L_i \|x-x_{\circ}\|~ ~ ~\forall ~x \in B(x_{\circ},\delta_i)\cap \Omega. $$
 Choose $\alpha_i > L_i$ and $\displaystyle \delta= \min_{1\leqslant i \leqslant p} \delta_i$, then for any $x \in \Omega \cap B(x_{\circ},\delta)$ the following can not hold
 \begin{eqnarray*}
\label{les}	
	f_i(x) &\leqslant& f_i(x_{\circ}) -\alpha_i \|x-x_{\circ}\|  ~~~\forall~  i=1,2,...,p.	
\end{eqnarray*}
Hence,  $x_{\circ}$ is  a local quasi efficient solution.}
\end{Remark}
\begin{Remark}
\label{rmrk-heart}{\rm
  The usual efficiency in scalar case coincides with the usual minimum. Thus, the quasi efficiency in scalar case, called quasi minimum would be a weaker generalized notion of the usual minimum. We claim that the quasi minimum is not a true generalization. Though, the quasi minimum implies lower semicontinuity. \\
  For,  let $x_\circ \in X$ be a  quasi minimum  (in the sense of Definition \ref{quasi-1}) for $(MP)$.  Then there exists $\alpha> 0$  such that for all $x \in  X$,
  \begin{eqnarray}
\label{lqes-core-1}	
	f(x) \geqslant  f(x_\circ) -\alpha  \|x-x_\circ\|.	
\end{eqnarray}
In view of Remark \ref{rmrk-3}, we see that the condition (\ref{lqes-core-1}) of the quasi minimum implies lower semicontinuity. However, if we consider the function $f: \mathbb R \rightarrow \mathbb R$ defined by $f(x):=x$ for $x \leqslant 0$ and $f(x):=1$ otherwise, then $f$ is neither continuous nor has a finite lower bound on $\mathbb R$, but has a quasi minimum at $x_{\circ}= 0$  for  $\alpha=1$. Thus, the quasi minimum does not imply continuity, and it does not imply finite lower bound also, hence disqualify to be WGM.}
\end{Remark}
   \begin{Remark}\rm
  It may be noted that \textit{$\epsilon-$efficiency} of \cite{Lorid84} and \textit{$\epsilon-$weak minima} of \cite{Deng} implies $\epsilon-$minima of the scalar case for some component of the vector function. In scalar case, both the notions coincide with $\epsilon-$minimum. Hence, they may qualify for weaker generalizations of the usual efficiency provided the component functions are lower semicontinuous. However, until now, it is not known how to characterize true weaker generalizations of the usual efficiency.
   \end{Remark}
 \begin{Remark}
\label{rmrk2} {\rm It is easy to see that
if $f_i$ are Lipschitz on $ \Omega$, then every point of $\Omega$ is a  quasi efficient solution of $(VP).$
One can easily observe that both the examples illustrated in Remark 3 of \cite{Bhatia06} have involved Lipschitz functions over $x \geqslant 0$. Hence, every point is a quasi efficient solution and hence a local quasi efficient  solution.}
\end{Remark}
\textbf{Hence, whether we use the criteria of WGM or not, we see that the quasi (or local quasi ) efficient solution is not a true generalization of efficient (or local efficient) solution.}\\

 In view of Remark \ref{rmrk1}, Theorem 2 (\textit{necessary optimality conditions}) of \cite{Bhatia06}, having a gap (a term of Clarke normal cone corresponding to the set constraint $X$ is  absent), can be corrected. The corrected version of Theorem 2 of \cite{Bhatia06} reads as follows.
\begin{Theorem}
\label{thm2}
Suppose that $x_{\circ} $  is   feasible point of  $(VP)$. Assume that $f_i, i =1,..p ~ {\rm and}  ~g_j, j=1,..m$  are locally Lipschitz at $x_{\circ}$. Then there exist $\alpha \in int(\mathbb R^p_+)$ and  $\lambda \in \mathbb R^p_+$  and  $\mu_j \in \mathbb R^m_+$ such that
\begin{eqnarray}
\label{eqn1}
0\in  \sum_{i=1}^p \lambda_i \partial f_i(x_{\circ}) + \sum_{j=1}^m \mu_j \partial g_j(x_{\circ}) +\sum_{i=1}^p \lambda_i \alpha_i \mathbb B + N_{X }(x_{\circ})\\
\label{eqn2}
~~~~{\rm and}~~~~
\mu_j g_j(x_{\circ})=0,~ \forall~ j =1,..m.
\end{eqnarray}
\end{Theorem}
\begin{proof}
In view of  Remark \ref{rmrk1}, there exist $\alpha \in int(\mathbb R^p_+)$ and $\delta > 0$ such that the following system has no solution
\begin{eqnarray*}
\label{les}	
	f_i(x) &\leqslant& f_i(x_{\circ}) -\alpha_i \|x-x_{\circ}\|  ~~~~  i=1,2,...,p,\\
	f_r(x) & < & f_r(x_{\circ})-\alpha_r\|x-x_{\circ}\|   ~~~{\rm for ~some ~}, r\\
	g_j(x) &\leqslant &0, ~~~~~~~~~j=1,..m\\
	x &\in& X \cap B(x_{\circ},\delta).
\end{eqnarray*}
In what follows,  $x_{\circ}$ is an efficient solution of the following subsidiary vector optimization problem $(VP')$ with locally Lipschitz data at $x_{\circ}$ and an abstract set constraint.
$$
\left.%
\begin{array}{lcl}
\min~f(x)=(f_1(x)+\alpha_1\|x-x_{\circ}\|,..,f_p(x)+\alpha_p\|x-x_{\circ}\|)\\
{\rm subject~ to }~g_j(x) \leqslant 0, j = 1,...m\\
x  \in X \cap B(x_{\circ},\delta).
\end{array}%
\right\}~~~~~~~~~~~~~~~~~~~~~~~~~~~~~(VP')
$$
 Noting that the $N_{X \cap B(x_{\circ},\delta)}(x_{\circ})= N_{X}(x_{\circ})$, the Fritz John necessary optimality conditions for $(VP')$ ensures the existence of
 $\lambda \in \mathbb R^p_+$  and  $\mu \in \mathbb R^m_+$ with $(\lambda, \mu) \neq 0$ such that
\begin{eqnarray}
\label{eqn3}
0\in  \sum_{i=1}^p   \lambda_i \partial ( f_i+\alpha_i\|x-x_{\circ}\|)  (x_{\circ})+ \sum_{j=1}^m \mu_j \partial g_j(x_{\circ}) + N_{X }(x_{\circ})\\
\nonumber {\rm and}~~~~\mu_j g_j(x_{\circ})=0,~ \forall~ j =1,..m.
 \end{eqnarray}
 Using the relation $ {\displaystyle  \partial ( \sum_{i=1}^p   s_i f_i)  (x_{\circ}) \subset  \sum_{i=1}^p   s_i \partial f_i (x_{\circ})},~~s_i \in \mathbb R$,  in (\ref{eqn3}), the result follows.
\end{proof}
\begin{Remark}\rm
\label{rmrk3}
Due to local Lipschitz continuity, the above Theorem holds for every feasible point $x$ in some neighbourhood of $x_{\circ}$.
\end{Remark}

\begin{Remark}\rm
\label{rmrk4}
    The proof of Theorem 3 of \cite{Bhatia06} uses Theorem 1 of \cite{Bhatia06} which is applicable if the function is lower semicontinuous and approximate convex. But in Theorem 3 of \cite{Bhatia06} lower semicontinuity is missing. Thus, in the statement of Theorem 3 of \cite{Bhatia06} assumption of lower semicontinuity must be added. However, the assumption of lower semicontinuity along with approximate convexity make functions locally Lipschitz, and hence invoking Remark \ref{rmrk1} the conclusion of Theorem 3 of \cite{Bhatia06} is trivial.
\end{Remark}
%
 \hspace*{0.5cm} In view of the above discussion, the  duality results are affected. In particular,  Theorem 5 of  \cite{Bhatia06} is applicable if the data is locally Lipschitz, which is missing in the assumption. However, considering the locally Lipschitz data  \cite[Theorem 5]{Bhatia06}  must be rechecked invoking Remark \ref{rmrk1} for local weak quasi efficiency \cite[Definition 9]{Bhatia06} and Theorem \ref{thm2}.
 \section{Concluding Remarks}
 In this note, first time, we have provided the criteria for weaker generalizations of the usual minimum. As a result, we have pointed out that the quasi efficiency \cite{Bhatia06} (having more than 16 citations in prestigious journals) is not appropriate and their results are either trivial or exaggerated. Further, we have corrected Theorem 2  of  \cite{Bhatia06}  as Theorem \ref{thm2} in the present note. A missing assumption in Theorem 3  of  \cite{Bhatia06} has been pointed out. We hope that the readers will be more attentive to use these and other generalizations without ignoring WGM criteria. \\\\
 \textbf{Acknowledgement:} The author is indebted to Professor Joydeep Dutta for insightful conversations.

\end{document}